\documentclass[pagebackref=true,10pt]{amsart}
\usepackage{hyperref}
\hypersetup{
    colorlinks,
    citecolor=blue,
    filecolor=blue,
    linkcolor=blue,
    urlcolor=blue 
}
\usepackage{CJKutf8} 
\usepackage{orcidlink}
\usepackage[margin=1.25 in]{geometry}
\usepackage{graphicx}%
\usepackage{multirow}%
\usepackage{amsfonts}%
\usepackage{textcomp}%
\usepackage{manyfoot}%
\usepackage{CJKutf8}
\usepackage{booktabs}%
\usepackage{algorithm}%
\usepackage{algorithmic}%
\usepackage{listings}%
\usepackage{xcolor}
\usepackage{tikz}
\usepackage{tikz-cd}
\usepackage{tikz-3dplot}
\usepackage{amsmath,mathrsfs,amssymb,amsthm,amscd}
\usepackage{mathtools}
\usepackage{url}
\usepackage{natbib}
\usepackage{epigraph}
\usepackage[all]{xy}
\usepackage{wasysym}
\usepackage{cleveref}

\newtheorem*{rep@theorem}{\rep@title}
\newcommand{\newreptheorem}[2]{%
\newenvironment{rep#1}[1]{%
 \def\rep@title{#2 \ref{##1} (restated)}%
 \begin{rep@theorem}}%
 {\end{rep@theorem}}}
\makeatother

\newtheorem{theorem}{Theorem}
\newtheorem*{theorem*}{Theorem}

\newtheorem{lemma}[theorem]{Lemma}
\newtheorem{proposition}[theorem]{Proposition}
\newtheorem{corollary}[theorem]{Corollary}
\newtheorem*{conjecture*}{Conjecture}
\newreptheorem{theorem}{Theorem}

\theoremstyle{definition}

\theoremstyle{remark}
\newtheorem*{remark}{Remark}

\DeclareMathOperator{\coker}{coker}

\numberwithin{equation}{section}

\begin{document}

\title{Bornes de torsion et un th\'eor\`eme effectif du pgcd}

\author{Hyuk Jun Kweon}
\address{Department of Mathematics, University of Georgia, USA}
\email{kweon@uga.edu}

\author{Madhavan Venkatesh}
\address{Department of Computer Science \& Engineering, IIT Kanpur, India}

\email{madhavan@cse.iitk.ac.in}

\date{}

\dedicatory{}

\begin{abstract}
We prove an effective, probabilistic version of Deligne's `th\'eor\`eme du pgcd' for a smooth, projective, geometrically integral (\textit{nice}) variety $X_{0}\subset \mathbb{P}^{N}$ over $\mathbb{F}_{q}$ of dimension $n$ and degree $D$, obtained via good reduction from a nice variety $\mathcal{X}_{0}$ over a number field $K$ at a prime $\mathfrak{p}\subset \mathcal{O}_{K}$. The main ingredients include bounding torsion in the Betti cohomology of $\mathcal{X}_{0}$, a mod -- $\ell$ big monodromy result and equidistribution of Frobenius in the representation associated to the sheaf of vanishing cycles modulo $\ell$.
\end{abstract}

\maketitle

\section{Introduction}
 Given a nice variety $X_0$ over $\mathbb{F}_{q}$ of dimension $n$ and degree $D$ obtained via good reduction, we write $X:=X_{0}\otimes \overline{\mathbb{F}}_{q}$ and $P_{i}(X/\mathbb{F}_{q}, T):=\det\left(1-TF_{q}^{\star} \ \vert \ \mathrm{H}^{i}(X, \mathbb{Q}_{\ell})\right)$, where $\ell$ is a prime not dividing $q$. Let $(X_{t})_{t\in\mathbb{P}^{1}}$ be a Lefschetz pencil of hyperplane sections\footnote{obtained after taking a Veronese re-embedding of degree 3}  on $X$. Denote by $Z\subset \mathbb{P}^{1}$ the finite set of nodal fibres and by $U=\mathbb{P}^{1}\setminus Z$, the subscheme parameterising the smooth fibres. As a consequence of the hard-Lefschetz theorem for $X$, Deligne \cite[Th\'eor\`eme 4.5.1]{Weilii} showed the following.

\begin{theorem*}
     The polynomial $P_{n-1}(X/\mathbb{F}_{q}, T)$ is the least common multiple of all polynomials $$f(T)=\prod_{j}(1-\alpha_{j}T)\in \mathbb{C}[T],$$ satisfying the condition that for any $t\in U(\mathbb{F}_{q^{r}})$, the polynomial  $$f(T)^{(r)}:=\prod_{j}(1-\alpha_{j}^{r}T)$$ divides $P_{n-1}(X_{t}/\mathbb{F}_{q^{r}}, T)$.
\end{theorem*}
\begin{remark}
    Deligne's theorem even holds without assuming that $X$ can be lifted to characteristic zero.
\end{remark}
Treating the embedding dimension $N$ as constant, our main result is as follows.
\begin{theorem}
\label{thm:main}
   There exists a polynomial $\Phi(x)\in \mathbb{Z}[x]$ independent of $\deg X=D$, such that for any extension $\mathbb{F}_{Q}/ \mathbb{F}_{q}$ with
    $$
    [\mathbb{F}_Q: \mathbb{F}_{q}] > \Phi(D),
    $$
    we have for any $u_{1}, u_{2}\in U(\mathbb{F}_{Q})$ chosen uniformly at random,
    $$
    P_{n-1}(X/ \mathbb{F}_{Q}, T)=\mathrm{gcd}\left(P_{n-1}(X_{u_{1}}/\mathbb{F}_{Q}, T),  P_{n-1}(X_{u_{2}}/\mathbb{F}_{Q}, T) \right);
    $$
    with probability $>2/3$.
\end{theorem}
We can further recover $P_{n-1}(X/\mathbb{F}_{q}, T)$ from   $P_{n-1}(X/\mathbb{F}_{Q_{i}}, T)$ with $i\in \{1, 2\}$, for two suitably chosen $Q_{i}=q^{r_{i}}$ with $r_{i}=\mathrm{poly}(D\log q)$ following a recipe for cyclic resultants of Weil polynomials due to Kedlaya \cite[\S 8]{ked}. This leads to the following algorithmic consequence (considering the embedding dimension $N$ \textit{fixed}), as a result of applying the Lefschetz hyperplane theorem combined with an effective Bertini theorem for the existence of hyperplane sections \cite[Theorem 1]{ballico}.

\begin{corollary}
    There is a polynomial-time reduction for the zeta function computation of nice varieties (coming from number fields via good reduction) over finite fields to that of the middle cohomology.
\end{corollary}
\begin{remark}
    This reduction is polynomial time in both the degree $D$ of the variety and $\log q$, where $q$ is the size of the finite field.
\end{remark}

In the DPhil dissertation of Walker \cite[1.2.2]{walker}, the possibility of using Deligne's gcd theorem is discussed in the context of developing algorithms to compute the zeta function of smooth, projective varieties. By the weak-Lefschetz theorem, cohomology in degrees other than the middle band of $n-1$, $n$, $n+1$ maps isomorphically to the cohomology of a hyperplane section. Further, in \cite[Theorem 1.4]{rsv}, an algorithm was given to compute $P_{1}(T)$ for any smooth, projective variety by proving the effective gcd theorem in the surface case (the torsion bounds here are due to \cite{kweon2021bounds}), and reducing to known algorithms for curves. This present work is a generalisation to $n$ dimensions, in particular, handling both the cases of symplectic and orthogonal monodromy. In the light of \cite[Theorem 1.1]{surf}, our main theorem gives rise to algorithms to compute $P_{2}(T)$ for any smooth, projective variety as well.

Our proof strategy begins by finding a prime $\ell$ of reasonable size, for which the hard-Lefschetz theorem holds with $\mathbb{Z}/\ell \mathbb{Z}$-coefficients; which reduces to the condition of the integral $\ell$-adic cohomology groups being torsion free. To this regard, we first obtain torsion bounds in the characteristic zero Betti setting using cylindrical algebraic decomposition.

Choosing a torsion-free $\ell$, hard-Lefschetz modulo $\ell$ implies the irreducibility of the representation associated to the local system of vanishing cycles modulo $\ell$ on $U$. If the $\ell$-adic monodromy is infinite, this implies that the monodromy image is `big', using a result of Hall \cite{hall}. An equidistribution theorem of Katz \cite{kasar} then dictates the likelihood of two Frobenii having coprime characteristic polynomials, which we make precise by bounding the error term therein.

\section{Torsion Bounds on Cohomology Groups}

The aim of this section is to give explicit upper bounds on the order of the torsion subgroups of cohomology groups. The bound is singly exponential in the degree of the defining polynomials and triply exponential in the dimension of the ambient projective space. To obtain these upper bounds, we will use a regular cellular decomposition of the variety. The number of cells will then provide an upper bound on the order of the torsion subgroups. The main tool for finding such a cellular decomposition is cylindrical algebraic decomposition, introduced by Collins \cite{collins}.

\begin{theorem}\label{thm:cell bound}
  Let $X \subset \mathbb{R}^N$ be a compact real algebraic variety defined by $m$ polynomials of degree $\leq d$. Then there is a regular cell complex, with number of cells at most 
    \[(2d)^{3^{N+1}} m^{2^N}. \]
\end{theorem}
\begin{proof}
  Collins' algorithm computes a cylindrical algebraic decomposition of $X$ with at most $(2d)^{3^{N+1}} m^{2^N}$ cells \cite[Theorem~12]{collins}. Although this may not yield a regular cellular decomposition \cite[Example~2.1]{davenport}, performing a generic linear change of coordinates before running the algorithm ensures that the cylindrical algebraic decomposition becomes a regular cell complex \cite[Theorem~2]{schwartz}.
\end{proof}

The theorem above depends on the number $m$ of polynomials defining the variety $X$. This is bounded by the number of monomials of degree $\leq d$, meaning that
\[m \leq \binom{N+d}{N}.\]

\begin{lemma}\label{lem:torsion bound}
  Let $M$ be an $m \times n$ matrix representing a linear transformation 
  \[\varphi \colon \mathbb{Z}^n \to \mathbb{Z}^m.\] 
  Suppose that all entries of $M$ are either $-1$, 0, or 1. Then 
  \[\# (\coker \varphi)_{\mathrm{tors}} \leq \min \{ m!, n! \}.\]
\end{lemma}
\begin{proof}
  Let $D$ be the Smith Normal Form of $M$, with diagonal entries $d_0, d_1, \dots, d_{r-1}$. Then
  \begin{align*}
      (\coker \varphi)_{\mathrm{tors}} &\simeq \mathbb{Z} / d_0 \mathbb{Z} \oplus \mathbb{Z} / d_1 \mathbb{Z} \oplus \dots \oplus \mathbb{Z} / d_{r-1} \mathbb{Z} \\
      \# (\coker \varphi)_{\mathrm{tors}} &= d_0 d_1 \cdots d_{r-1}.
  \end{align*}
  Moreover, $d_0 d_1 \cdots d_{r-1}$ is the greatest common divisor of the determinants of all $r \times r$ minors of $M$. Since the Leibniz expansion for such a minor consists of $r!$ terms,
  \[d_0 d_1 \cdots d_{r-1} \leq r! \leq \min \{ m!, n! \}. \qedhere\]
\end{proof}

Now, \Cref{thm:cell bound} together with \Cref{lem:torsion bound} gives the following theorem.

\begin{theorem}\label{thm:real affine torsion bound}
    Let $X \subset \mathbb{R}^N$ be a compact real algebraic variety defined by polynomials of degree $\leq d$. Then
    \[ \# \mathrm{H}^i_B(X, \mathbb{Z})_{\mathrm{tors}} \leq \left((2d)^{3^{N+1}} \binom{N+d}{N}^{2^N}\right)!. \]
\end{theorem}
\begin{remark}
    We denote Betti cohomology with $\mathrm{H}^{i}_{B}$ and \'etale cohomology with $\mathrm{H}^{i}$.
\end{remark}
The above theorem applies only when $X$ is a real affine variety, and the set of its $\mathbb{R}$-points is compact. We aim to obtain a similar bound for the case where $X$ is a complex projective variety. This can be achieved by using the standard embedding $\mathbb{CP}^N \rightarrow \mathbb{C}^{(N+1)^2}$ and dividing each complex coordinate into two real coordinates.

\begin{theorem}\label{thm:complex projective torsion bound}
    Let $X \subset \mathbb{CP}^N$ be a complex projective variety defined by homogeneous polynomials of degree $\leq d$. Then
    \[ \# \mathrm{H}^i_B(X, \mathbb{Z})_{\mathrm{tors}} \leq \left((2d)^{3^{(N+1)^2+1}} \binom{(N+1)^2+d}{(N+1)^2}^{2^{(N+1)^2}}\right)!. \]
\end{theorem}
\begin{proof}
    Recall that the standard embedding $ \mathbb{CP}^N \rightarrow \mathbb{C}^{(N+1)^2} $ is given by
    \[(z_0 : z_1 : \cdots : z_N) \mapsto \frac{1}{\sum_{i=0}^N |z_i|^2} 
    \begin{pmatrix} 
      z_0 \overline{z_0} & z_0 \overline{z_1} & \cdots & z_0 \overline{z_N} \\ 
      z_1 \overline{z_0} & z_1 \overline{z_1} & \cdots & z_1 \overline{z_N} \\ 
      \vdots & \vdots & \ddots & \vdots \\ 
      z_N \overline{z_0} & z_N \overline{z_1} & \cdots & z_N \overline{z_N}
    \end{pmatrix}.\]
    The image is defined by polynomials of degree $\leq 2$. A hypersurface in $\mathbb{CP}^N$ defined by a homogeneous polynomial $f$ can be expressed by several polynomials of the same degree in $\mathbb{C}^{(N+1)^2}$. Since the image of the embedding is a Hermitian matrix, half of the real coordinates can be reconstructed from the other half. Thus, applying \Cref{thm:real affine torsion bound} yields the desired result.
\end{proof}

\begin{corollary}\label{cor:simple torsion bound}
    Let $X \subset \mathbb{CP}^N$ be a complex projective variety defined by homogeneous polynomials of degree $\leq d$. Then
    \[ \# \mathrm{H}^i_B(X, \mathbb{Z})_{\mathrm{tors}} \leq 2^{d^{2^{3N^2}}}. \]
\end{corollary}
\begin{proof}
    We may assume that $d \geq 2$ and $N \geq 4$, because projective spaces, hypersurfaces and curves do not have torsion in their cohomology groups. For simplicity let $M = N+1$ and
    \[L = (2d)^{3^{M^2+1}} \binom{M^2+d}{M^2}^{2^{M^2}}. \]
    Since
    \[ \binom{M^2+d}{M^2} \leq (M^2+d)^{M^2} \leq \left(d^M\right)^{M^2} = d^{M^3}, \]
    we obtain
    \[L \leq \left(d^2\right)^{3^{M^2+1}} \left(d^{M^3}\right)^{2^{M^2}} \leq d^{2\cdot 3^{M^2+1} + M^3 2^{M^2}}.\]
    As a result,
    \[\log_d \log_2 L! \leq 2 \log_d L = 4 \cdot 3^{M^2+1} + M^3 2^{M^2} \leq  2^{3 (M-1)^2}. \qedhere\]
\end{proof}

\begin{corollary}
\label{cor:ell}
  Let $X \subset \mathbb{CP}^N$ be a complex projective variety defined by homogeneous polynomials of degree at most $d$. Then there exists a prime number \[\ell \leq d^{2^{4N^2}}\] such that $\mathrm{H}^i_B(X, \mathbb{Z})$ is torsion-free for all $i$.
\end{corollary}
\begin{proof}
    By \Cref{cor:simple torsion bound}, 
    \[ \# \prod_{i=0}^N \mathrm{H}^i_B(X, \mathbb{Z}) < \left(2^{d^{2^{3N^2}}}\right)^N = 2^{N d^{2^{3N^2}}}. \]
    Therefore, there exists a prime number $\ell$ among the first 
    \[k = N d^{2^{3N^2}} \] 
    primes such that $\prod_{i=0}^N \mathrm{H}^i_B(X, \mathbb{Z})$ is $\ell$-torsion free. Since $k \geq 4$, \cite[Theorem 2]{rosser} implies that the $k$-th prime number is smaller than
    \[ k (\log k + 2 \log \log k) \leq k^2 \leq \left(N d^{2^{3N^2}}\right)^2 \leq d^{2^{4N^2}}.\qedhere \]
\end{proof}

The sum of the Betti numbers of $X$ has an upper bound that is polynomial in $d$ and singly exponential in $N$ 
\cite[Corollary 2]{milnor}.
\begin{theorem}[Milnor]
\label{thm:milnor}
    Let $X \subset \mathbb{CP}^N$ be a complex projective variety defined by homogeneous polynomials of degree $\leq d$. Then
    \[ \sum_{i\geq 0} \operatorname{rank} \mathrm{H}^i_B(X, \mathbb{Z}) \leq Nd(2d - 1)^{2N+1}. \]
\end{theorem}
This bound is derived by bounding the number of critical points of a Morse function. Since a Morse cohomology is generated by these critical points, the number of generators of the torsion subgroups is also bounded by the same value. Thus, if the order of each generator is not excessively large, we expect to obtain an upper bound on the order of $\mathrm{H}^i_B(X,\mathbb{Z})_{\mathrm{tors}}$ that is singly exponential in $d$ and doubly exponential in $N$. However, determining the boundary map in Morse homology requires solving differential equations arising from a pseudo-gradient field, and these solutions do not form a semi-algebraic set. This is the technical reason why it is difficult to derive a bound doubly exponential in $N$.

Further, as we are in the realm of complex, smooth, projective varieties, one may also look at other methods towards obtaining such bounds for torsion. Note firstly, using the K\"unneth formula, that it suffices to bound torsion in cohomology in even degree. Next, torsion therein can be of two types, algebraic or transcendental. Guaranteed that the torsion is algebraic, it may be possible to bound it using the connected components of the Chow variety of $X$. Examples with transcendental torsion seem to have the order depend on the degree of the variety in question (see \cite[Theorem 3]{soulevoisin} for concrete examples using Godeaux surfaces). This line of work, involving explicitly constructing transcendental torsion algebraic cycles began with Atiyah and Hirzebruch \cite{atiyahirz}, who thereby provided counterexamples to the integral Hodge conjecture. One is led to conjecture that the torsion coming from transcendental cycles can likewise be controlled uniformly by the degree of the variety.

Over fields of positive characteristic, Gabber's theorem \cite{gabber} guarantees the torsion-freeness of the integral $\ell$-adic \'etale cohomology groups for all but finitely many $\ell$, so one is tempted to make the analogous conjecture over arbitrary base fields as well.

\begin{conjecture*}
There exist polynomials $\psi(x), \phi(x) \in \mathbb{Z}[x]$ such that for any smooth, projective variety $X\subset \mathbb{P}^{N}$ of dimension $n$ and degree $D$ over an algebraically closed field $k$, we have
$$
\mathrm{H}^{i}(X, \mathbb{Z}_{\ell})_{\mathrm{tors}}=0
$$
for $0\leq i \leq 2n$, when $$\ell>\psi(D^{\phi(N)})$$
is a prime number coprime to the characteristic of $k$.
\end{conjecture*}

\section{Monodromy}
In this section, we recall the notion of monodromy in the context of a Lefschetz pencil of hyperplane sections on a smooth, projective variety. The main objective is to show that the $\mathrm{mod}$-$\ell$ monodromy is as large as possible for primes $\ell$ of a reasonable size. 

Let $X$ be a nice variety satisfying our main assumptions. We may fibre $X$ as a Lefschetz pencil of hyperplane sections $\pi:\tilde{X}\rightarrow \mathbb{P}^{1}$, where $\tilde{X}$ is the variety obtained by blowing up $X$ at the axis of the pencil, and the fibres of $\pi$ are the hyperplane sections. Denote by $U\subset \mathbb{P}^{1}$ the locus of smooth fibres and by $Z:=\mathbb{P}^{1}\setminus U$, the finite set parameterising the nodal fibres. Let $\ell$ be coprime to $q$. Consider the constructible sheaf $\mathcal{F}:=R^{n-1}\pi_{\star}\mathbb{Q}_{\ell}$ on $\mathbb{P}^{1}$. The restriction $\mathcal{F}\vert_{U}$ defines a local system on $U$, and we can speak of the monodromy action of the geometric \'etale fundamental group $\pi_{1}(U, u)$, where $u\rightarrow U$ is a geometric point. We know further, that the tame fundamental group $\pi_{1}^{\mathfrak{t}}(U, u)$ is topologically generated by $\#Z$ elements $\sigma_{i}$ satisfying the relation $\prod_{i}\sigma_{i}=1$. Moreover, for each $z\in Z$, one obtains a vanishing cycle $\delta_{z}\in \mathrm{H}^{n-1}(X_{\overline{\eta}}, \mathbb{Z}/\ell \mathbb{Z})$ via the exact sequence
\[
\begin{tikzcd}
0\arrow{r} & \mathrm{H}^{n-1}(X_{z}, \mathbb{Z}/\ell \mathbb{Z}) \arrow{r} & \mathrm{H}^{n-1}(X_{\overline{\eta}}, \mathbb{Z}/\ell \mathbb{Z}) \arrow{r} & \mathbb{Z}/\ell\mathbb{Z}
\end{tikzcd}
\]
with the final arrow being given by $\gamma\mapsto \langle \gamma, \delta_{z}\rangle$, where 
$$
\langle \cdot, \cdot \rangle : \mathrm{H}^{n-1}(X_{\overline{\eta}}, \mathbb{Z}/\ell \mathbb{Z})\times \mathrm{H}^{n-1}(X_{\overline{\eta}}, \mathbb{Z}/\ell \mathbb{Z}) \longrightarrow \mathbb{Z}/\ell \mathbb{Z}
$$
is the Poincar\'e duality pairing. Furthermore, $\delta_{z}$ is unqiuely determined up to sign by the Picard-Lefschetz formulas
\begin{equation}\label{eqn:piclef}
    \sigma_{z}(\gamma)=\gamma\pm\epsilon_{z}\cdot \langle \gamma, \delta_{z}\rangle \cdot  \delta_{z},
\end{equation}
where for a uniformising parameter $\theta_{z}$ at $z$, we have $\sigma_{z}(\theta_{z}^{1/\ell})=\epsilon_{z}\theta_{z}^{1/\ell}$.
In the limit, we obtain an integral $\ell$ -- adic vanishing cycle in $\mathrm{H}^{n-1}(X_{\overline{\eta}}, \mathbb{Z}_{\ell})$ which is defined up to torsion, and becomes unique up to sign upon tensoring with $\mathbb{Q}_{\ell}$. We denote by $\mathcal{E}_{\overline{\eta}}$ the space generated by all the vanishing cycles $\delta_{z}$\footnote{abusing notation} for $z\in Z$ in $\mathrm{H}^{n-1}(X_{\overline{\eta}}, \mathbb{Q}_{\ell})$ and by $\mathcal{E}_{u}$ for $u\in U$, the image of $\mathcal{E}_{\overline{\eta}}$ under the specialisation isomorphism $\mathcal{F}_{\overline{\eta}}\rightarrow \mathcal{F}_{u}$.
\\ \\
By the hard-Lefschetz theorem \cite[Theorem 4.3.9]{Weilii}, we have for $u\in U$,
\begin{equation}
    \mathcal{F}_{u}\simeq \mathrm{H}^{n-1}(X_{u}, \mathbb{Q}_{\ell})\simeq \mathrm{H}^{n-1}(X, \mathbb{Q}_{\ell})\oplus \mathcal{E}_{u}\;,
\end{equation}
where $\mathcal{E}_{u}$ is the space of vanishing cycles at $u$. In particular, $$\mathrm{H}^{n-1}(X, \mathbb{Q}_{\ell})=\mathrm{H}^{n-1}(X_{u}, \mathbb{Q}_{\ell})^{\pi_{1}(U, u)}=\mathcal{E}_{u}^{\perp}\;,$$ with respect to the Poincar\'e duality pairing on $\mathrm{H}^{n-1}(X_{u}, \mathbb{Q}_{\ell})$ and $\mathcal{E}_{u}\cap \mathcal{E}_{u}^{\perp}=0$. Further, the sheaf $\mathcal{F}\vert_{U}$ decomposes as
$$
\mathcal{F}\vert_{U}\simeq \underline{\mathcal{V}} \oplus \mathcal{E}
$$
where $\underline{\mathcal{V}}$ is the constant sheaf on $U$ associated to $\mathrm{H}^{n-1}(X, \mathbb{Q}_{\ell})$ and $\mathcal{E}$ is the sheaf of vanishing cycles. The sheaf $\mathcal{E}$ is locally constant on $U$ of rank, say, $r\in \mathbb{Z}_{\geq 0}$. Write $\mathcal{E}^{\mathbb{Z}_{\ell}}$ for the sheaf of integral $\ell$ -- adic vanishing cycles and denote by $\mathcal{E}^{\ell}:=\mathcal{E}^{\mathbb{Z}_{\ell}}{\otimes}\mathbb{F}_{\ell}$ the sheaf of mod -- $\ell$ vanishing cycles. We begin by showing the following.

\begin{lemma}\label{lemtr}
Let $\ell$ be a prime coprime to $q$, such that the cohomology groups $\mathrm{H}^{i}(X, \mathbb{Z}_{\ell})$ are all torsion-free for $0\leq i \leq 2n$. Let $X_{u}$ be a smooth hyperplane section of $X$ from the above Lefschetz pencil. Then the cohomology groups $\mathrm{H}^{j}(X_{u}, \mathbb{Z}_{\ell})$ for $0\leq j\leq 2n-2$ are all torsion free.
\end{lemma}
\begin{proof}
    By the Lefschetz hyperplane theorem\footnote{also known as the weak-Lefschetz theorem}, we know that the induced map $\mathrm{H}^{j}(X, \mathbb{Z}_{\ell})\rightarrow \mathrm{H}^{j}(X_{u}, \mathbb{Z}_{\ell})$ is an isomorphism for $j<n-1$. Moreover, we also know, by Poincar\'e duality, that the Gysin map $\mathrm{H}^{j}(X_{u}, \mathbb{Z}_{\ell})\rightarrow \mathrm{H}^{j+2}(X, \mathbb{Z}_{\ell})$ is an isomorphism for $j>n-1$. It remains to show that $\mathrm{H}^{n-1}(X_{u}, \mathbb{Z}_{\ell})$ is torsion-free. We recall the universal coefficient theorem for the affine variety $X\setminus X_{u}$ on cohomology with compact support
    \begin{equation}
        \label{eqn:uctdgm}
    \mathrm{H}^{n-1}_{c}(X\setminus X_{u}, \mathbb{Z}/\ell \mathbb{Z})=\left(\mathrm{H}^{n-1}_{c}(X\setminus X_{u}, \mathbb{Z}_{\ell})\otimes \mathbb{Z}/\ell\mathbb{Z}\right) \oplus \mathbf{Tor}_{1}^{\mathbb{Z}_{\ell}}\left(\mathrm{H}^{n}_{c}(X\setminus X_{u}, \mathbb{Z}_{\ell}), \mathbb{Z}/\ell \mathbb{Z}\right).
    \end{equation}
By Artin vanishing and Poincar\'e duality, we know $\mathrm{H}^{n-1}_{c}(X\setminus X_{u}, \mathbb{Z}/\ell \mathbb{Z})=0$, so we have from (\ref{eqn:uctdgm}) that $\mathrm{H}^{n}_{c}(X\setminus X_{u}, \mathbb{Z}_{\ell})$ is torsion-free. Therefore, from  the relative long exact sequence associated to the pair $(X, X\setminus X_{u})$,
    \begin{equation}
        \label{eqn:rel}
    \ldots \rightarrow \mathrm{H}^{j}_{c}(X\setminus X_{u}, \mathbb{Z}_{\ell})\rightarrow \mathrm{H}^{j}(X, \mathbb{Z}_{\ell})\rightarrow \mathrm{H}^{j}(X_{u}, \mathbb{Z}_{\ell})\rightarrow\ldots
    \end{equation}
     we see that $$\mathrm{H}^{n-1}(X_{u}, \mathbb{Z}_{\ell})/\mathrm{H}^{n-1}(X, \mathbb{Z}_{\ell})$$ is torsion-free. We conclude the proof using the torsion-freeness assumption on $\mathrm{H}^{n-1}(X, \mathbb{Z}_{\ell})$. 
\end{proof}

\begin{lemma}
\label{lem:modhardlef}
    Let $\ell$ be a prime coprime to $q$, such that the cohomology groups $\mathrm{H}^{i}(X, \mathbb{Z}_{\ell})$ are all torsion-free for $0\leq i \leq 2n$ and let $X_{u}$ be a hyperplane section of $X$ from the above Lefschetz pencil. Then, the hard-Lefschetz theorem holds modulo $\ell$, i.e., we have
    \begin{equation}
    \mathrm{H}^{n-1}(X_{u}, \mathbb{Z}/\ell \mathbb{Z})\simeq \mathrm{H}^{n-1}(X, \mathbb{Z}/\ell \mathbb{Z})\oplus \mathcal{E}^{\ell}_{u}.
    \end{equation}
\end{lemma}
\begin{proof}
    From the diagram \cite[(4.3.3.2)]{Weilii}, we see that the exact sequence
    $$
    0\rightarrow \mathcal{E}^{\mathbb{Z}_{\ell}}_{u}\rightarrow \mathrm{H}^{n-1}(X_{u}, \mathbb{Z}_{\ell})\rightarrow \mathrm{H}^{n+1}(X, \mathbb{Z}_{\ell})\rightarrow 0
    $$
    splits as the terms involved are all torsion-free. Next, one notices that the hard-Lefschetz map
    $$
   \lambda : \mathrm{H}^{n-1}(X, \mathbb{Z}_{\ell})\rightarrow \mathrm{H}^{n+1}(X, \mathbb{Z}_{\ell})
    $$
    obtained by taking cup-product with the class of $X_{u}$ is injective by the hard-Lefschetz theorem and the fact that $\mathrm{H}^{n-1}(X, \mathbb{Z}_{\ell})$ is torsion-free. The map is also surjective as we know $$\mathrm{H}^{n-1}(X_{u}, \mathbb{Z}_{\ell})/\mathrm{H}^{n-1}(X, \mathbb{Z}_{\ell})$$ is torsion-free. Further, we note that $\mathrm{H}^{n-1}(X, \mathbb{Z}_{\ell})\cap \mathcal{E}^{\mathbb{Z}_{\ell}}_{u}\subset \mathrm{H}^{n-1}(X_{u}, \mathbb{Z}_{\ell})_{\mathrm{tors}}=0$, by Lemma~\ref{lemtr}. Therefore, we have
    $$
    \mathrm{H}^{n-1}(X_{u}, \mathbb{Z}_{\ell})\simeq \mathrm{H}^{n-1}(X, \mathbb{Z}_{\ell})\oplus \mathcal{E}^{\mathbb{Z}_{\ell}}_{u}.
    $$
    Tensoring by $\mathbb{Z}/\ell \mathbb{Z}$ and using torsion-freeness once more gives the result. 
\end{proof}

\begin{lemma}[Irreducibility]
\label{lem:irr}
    The representation $\rho_{\ell}:\pi_{1}(U, u)\rightarrow \mathrm{GL}(r, \mathbb{Z}/\ell \mathbb{Z})$ associated to the local system $\mathcal{E}^{\ell}$ of mod -- $\ell$ vanishing cycles on $U$ is irreducible.
\end{lemma}
\begin{proof}
    Let $W$ denote the representation corresponding to the mod -- $\ell$ vanishing cycles $\mathcal{E}^{\ell}_{u}$ and let $W'\subset W$ be a subspace fixed under the action of $\pi_{1}(U, u)$. Let $\gamma \in W'$ be such that $ \gamma \ne 0$. We claim firstly that $\langle \gamma, \delta_{z}\rangle \ne 0$ for a vanishing cycle $\delta_{z}$ for some $z\in Z$. Otherwise, we would have $\gamma\in W^{\perp}\cap W$, which is trivial by Lemma~\ref{lem:modhardlef}. In particular, by the Picard-Lefschetz formula (\ref{eqn:piclef}), we have $\sigma_{z}(\gamma)-\gamma=\langle \gamma, \delta_{z}\rangle \cdot \delta_{z}\in W'$, implying $\delta_{z}\in W'$. However, by \cite[Theorem 5.2]{illusie}, the vanishing cycles are all conjugate under the action of $\pi_{1}(U, u)$, so we must have $W'=W$.
\end{proof}

\begin{theorem}[Big monodromy]
\label{thm:bigm}
    Assume the sheaf $\mathcal{E}^{\mathbb{Z}_{\ell}}$ has big monodromy, i.e., the associated representation $\rho:\pi_{1}(U, u)\rightarrow \mathrm{GL}(\mathcal{E}^{\mathbb{Z}_{\ell}}_{u})$ has Zariski dense image in the corresponding symplectic or orthogonal groups. Then the sheaf $\mathcal{E}^{\ell}$ has big monodromy, i.e., the mod -- $\ell$ representation $\rho_{\ell}: \pi_{1}(U, u)\rightarrow \mathrm{GL}(r, \mathbb{Z}/\ell \mathbb{Z})$ has maximal image. In particular, if $n$ is even, then $\mathrm{im}(\rho_{\ell})= \mathrm{Sp}(r, \mathbb{Z}/\ell \mathbb{Z})$ and if $n$ is odd, $\mathrm{im}(\rho_{\ell})$ is one of the following subgroups of the orthogonal group $\mathrm{O}(r, \mathbb{Z}/\ell\mathbb{Z})$
    \begin{itemize}
        \item [(a)] the kernel of the spinor norm,
        \item [(b)] the kernel of the product of the spinor norm and the determinant map,
        \item[(c)] the full orthogonal group.
    \end{itemize}
    
\end{theorem}
\begin{proof}
We intend to apply \cite[Theorem 3.1]{hall} to $W$. Assume firstly that $n$ is even. In this case, the Poincar\'e duality pairing is alternating and $W$ is even-dimensional. Then, the elements $\rho_{\ell}(\sigma_{i})$ act via the Picard-Lefschetz formulas (\ref{eqn:piclef}) as transvections on $W$. Using the irreducibility from Lemma~\ref{lem:irr}, we may conclude that the image of $\rho_{\ell}$ is the full symplectic group $\mathrm{Sp}(r, \mathbb{F}_{\ell})$.

In the case $n$ is odd, the pairing is symmetric, so the monodromy is orthogonal. Here, the Picard-Lefschetz formulas act by reflections, in particular, even as isotropic shears. We again appeal to \cite[Theorem 3.1]{hall} to conclude that the geometric mod -- $\ell$ monodromy must be one of the subgroups of the orthogonal group of index at most two (other than the special orthogonal group), as listed above. 
\end{proof}
\begin{remark}
    We note that using work of Katz \cite[Theorem 2.2.4]{katzlarsen}, we may assume that $\mathcal{E}^{\mathbb{Z}_{\ell}}$ has big monodromy always (i.e., its image is infinite), at the cost of a Veronese embedding of constant degree.
    \end{remark}


\section{Estimates in algebraic groups}

In this section, we obtain probability estimates in order to prove our main Theorem~\ref{thm:main}. Specifically, we investigate the likelihood of a matrix, chosen uniformly in symplectic or orthogonal similitude groups having characteristic polynomial coprime to a given one of the respective type. 

\subsection{Symplectic monodromy}
We begin with the case where $n=\dim X$ is even, so the monodromy is symplectic. Consider the symplectic group $\mathrm{Sp}(s, \mathbb{F}_{\ell})$, where $\ell$ is a prime and $s=2r$. We have the exact sequence
$$
1\rightarrow \mathrm{Sp}(s, \mathbb{F}_{\ell})\rightarrow \mathrm{GSp}(s, \mathbb{F}_{\ell})\rightarrow \mathbb{F}_{\ell}^{*}\rightarrow 1
$$
where $\mathrm{GSp}(r, \mathbb{F}_{\ell})$ is the group of symplectic similitudes. Let $\lambda\in \mathbb{F}_{\ell}^{*}$ and write by $\mathrm{GSp}(r, \mathbb{F}_{\ell})^{\lambda}$, the conjugacy class of similitudes with multiplicator $\lambda$. The following is the set of all possible (reversed) characteristic polynomials of symplectic similitudes with multiplier $\lambda$
$$
M^{\lambda}_{r}:=\{f(T)=1+a_{1}t+\ldots a_{2r-1}t^{2r-1}+a_{2r}t^{2r} \ \vert \ a_{i}\in \mathbb{F}_{\ell}, \ a_{2r-i}=\alpha^{r-i}a_{i}, \ 0\leq i\leq r\}.
$$
\begin{proposition}
    \label{symplectic}
 Let $f(T)$ be the characteristic polynomial of a matrix in $\mathrm{GSp}(2r, \mathbb{F}_{\ell})^{\lambda}$ for some $\lambda\in \mathbb{F}_{\ell}^{*}$. Denote by $C\subset \mathrm{GSp}(2r, \mathbb{F}_\ell)$ the set of matrices with characteristic polynomial not coprime with $f(T)$. Then for $\ell>119r^{2}$,
$$
\frac{\#\left(C\cap \mathrm{GSp}(2r, \mathbb{F}_{\ell})^{\lambda}\right)}{\# \mathrm{Sp}(2r, \mathbb{F}_{\ell})} \;\leq\; 1/4\;.
$$

\end{proposition}
\begin{proof}
    This is \cite[Lemma 3.10]{rsv}.
\end{proof}

\subsection{Orthogonal monodromy}
We are now concerned with the case when $n=\dim X$ is odd. In particular, we have that the action of Frobenius on $\mathrm{H}^{n-1}(X_{u}, \mathbb{Z}/\ell \mathbb{Z})$ is via an orthogonal similitude, i.e., the image  $\rho_{\ell}(\pi_{1}(U_{0}, u))\subset \mathrm{GO}(\mathrm{V})$, where $\mathrm{V}$ is the subspace $\mathcal{E}^{\ell}_{u}\subset \mathrm{H}^{n-1}(X_{u}, \mathbb{Z}/\ell \mathbb{Z})$ of dimension $s$, regarded as an $\mathbb{F}_{\ell}$ -- vector space. We begin by recalling the well-known bounds for the size of the orthogonal group.

\begin{lemma}
    We have
    $$
  2\ell^{2r^{2}}(\ell-1)^{r} \leq  \#\mathrm{O}(2r+1, \mathbb{F}_{\ell})= 2\ell^{r^{2}}\prod_{i=1}^{r}(\ell^{2i}-1)\leq 2\ell^{2r^{2}+r}
    $$ and
    $$
 \ell^{2r^{2}}(\ell-1)^{r}   \leq \#\mathrm{O}(2r, \mathbb{F}_{\ell})\leq 2\ell^{2r^{2}+r}
    $$
\end{lemma}

Let $
N_{r}^{\lambda}
$
now be the space of reciprocal polynomials of degree at most $s=2r$, or $s=2r+1$ in one variable, with multiplier $\lambda$ and coefficients in $\mathbb{F}_{\ell}$. 
We may identify it with the affine space $\mathbb{A}^{r}$. 
Like in the symplectic case, we have an exact sequence

\begin{equation}
\label{exac}
1\rightarrow \mathrm{O}(s, \mathbb{F}_{\ell})\rightarrow \mathrm{GO}(s, \mathbb{F}_{\ell})\rightarrow \mathbb{F}_{\ell}^{*}\rightarrow 1
\end{equation}

For $\lambda\in \mathbb{F}_{\ell}^{*}$, consider a map 
$$\Psi: \mathrm{GO}(s, \overline{\mathbb{F}}_{\ell})^{\lambda}\rightarrow \mathbb{A}_{\overline{\mathbb{F}}_\ell}^{r}$$
where a matrix is mapped to its (reversed) characteristic polynomial. The map $\Psi$ is a morphism of algebraic varieties.
We know that $\dim \mathrm{O}(s, \mathbb{F}_{\ell})=s(s-1)/2$. Given a polynomial $f(T)$ that we know is the characteristic polynomial of a matrix in $\mathrm{GO}(s, \mathbb{F}_{\ell})$, we seek to estimate the size of $\Psi^{-1}(W)\cap \mathrm{GO}(s, \mathbb{F}_{\ell})^{\lambda}$, where $W\subset \mathbb{A}^{r}$ parametrises those polynomials which have a factor common with $f(T)$. The map $\Psi$ is clearly surjective over $\overline{\mathbb{F}}_{\ell}$, so applying the theorem on fibre dimension, we see that generically, for $x\in \mathbb{A}^{r}$, we have
$$
\dim \Psi^{-1}(x)=s(s-1)/2 -r\leq 2r^{2}.
$$
We observe the following next.

\begin{lemma}
\label{orthlem}
The fibre dimension of $\Psi$ is $s(s-1)/2-r$ on the open subset $Y$ of $\mathbb{A}^{r}$ parametrising those characteristic polynomials with distinct roots. Moreover, writing $V=\mathbb{A}^{r}\setminus Y$, we have
 $$   \frac{\#V(\mathbb{F}_{\ell})}{\ell^{r}}\leq O(1/\ell)$$
 where the implied constant is independent of $\ell$ and depends linearly on $r$. Further,
 $$
 \frac{\#\Psi^{-1}(Y)(\mathbb{F}_{\ell})}{\mathrm{O}(s, \mathbb{F}_{\ell})}\geq 1-\Omega(1/\ell),
 $$
 where now, the implied constant is independent of $\ell$ and of the form $\exp(\mathrm{poly}(r))$.
\end{lemma}
\begin{proof}
For an element in $Y$, its fibre consists of a conjugacy class in $\mathrm{GO}(s, \mathbb{F}_{\ell})$ intersected with $\mathrm{GO}(s, \mathbb{F}_{\ell})^{\lambda}$. Elements in the fibre have distinct eigenvalues.  We see that a matrix $A$ in the fibre is stabilised by a maximal torus of dimension $r$, hence the fibre dimension here is minimised.

The complement, $V$, of $Y$ is a hypersurface in $\mathbb{A}^{r}$ of degree at most $8r$, obtained via the vanishing of the discriminant associated to a formal characteristic polynomial. We conclude the first estimate using \cite[pg 45]{bs}.

For the second estimate, we note that $\Psi^{-1}(V)$ is now a proper, closed subvariety of $\mathrm{GO}(s)^{\lambda}$ of degree $\exp(\mathrm{poly}(r))$ and codimension at least one. The number of its $\mathbb{F}_{\ell}$ -- rational points can be bounded via the Lang-Weil estimates \cite[Theorem 7.5]{cafure}, and can thus be avoided with high probability.
\end{proof}

\begin{proposition}
\label{orthprob}
 Let $f(T)\in Y(\mathbb{F}_{\ell}) \subset N_{r}^{\lambda}$ be the reversed characteristic polynomial of a matrix in $\mathrm{GO}(s, \mathbb{F}_{\ell})^{\lambda}$. Denote by $\Lambda$ the set of matrices in $\mathrm{GO}(s, \mathbb{F}_{\ell})^{\lambda}$ such that their reversed characteristic polynomial has a common factor with $f(T)$. Then $$\frac{\#\Lambda}{\# \mathrm{O}(s, \mathbb{F}_{\ell})}\leq O(1/\ell),$$ where the implied constant is independent of $\ell$ and of the form $ \exp({\mathrm{poly}(r)})$.
\end{proposition}
\begin{proof}
Given $f(T)$, let $W_f\subset \mathbb{A}^{r}$ parametrise those polynomials which have a factor common with $f(T)$. It is a hypersurface, given by the vanishing of the formal resultant with $f(T)$ (see \cite[\S 3.3]{rsv}). Then, the set $\Lambda$ is just the set of $\mathbb{F}_{\ell}$ -- rational points of $\Psi^{-1}(W_f)\subset \mathrm{GO}(s)$, which is a proper, closed subvariety of degree at most $r^{\mathrm{poly}(r)}$. Then, we may conclude by the Lang-Weil estimates \cite[Theorem 7.5]{cafure} applied to $\Psi^{-1}(W_f)$.
\end{proof}

\section{Proof of Theorem~\ref{thm:main}}
We begin by recalling a version of Deligne's equidistribution theorem \cite{Weilii} due to Katz \cite[Theorem 9.7.13]{kasar}.
Let $U_{0}/\mathbb{F}_{q}$ be a smooth, affine, geometrically irreducible curve. Let $U$ be the base change to the algebraic closure. Pick a geometric point $u\rightarrow U$, lying over a closed point $u_{0}\in U(\mathbb{F}_{q})$ and denote by $\overline{\pi}_{1}:=\pi_{1}(U, u)$ the geometric \'etale fundamental group. Let $\pi_{1}$ denote the arithmetic fundamental group $\pi_{1}(U_{0}, u)$. For any closed point $v\in U(\mathbb{F}_{q})$, there exists an element $F_{q, v}\in \pi_{1}$ well-defined up to conjugacy, called the \textit{Frobenius element} at $v$. It is defined as follows. Writing $v=\mathrm{Spec}(\mathbb{F}_{q})\rightarrow U$, we obtain an induced map of fundamental groups
$$
\mathrm{Gal}(\overline{\mathbb{F}}_{q}/ \mathbb{F}_{q})\rightarrow \pi_{1}(U_{0}, v) \simeq \pi_{1}.
$$
The element $F_{q, v}\in \pi_{1}$ is simply the image in $\pi_{1}$ of the Frobenius element in $\mathrm{Gal}(\overline{\mathbb{F}}_{q}/ \mathbb{F}_{q})$ under the composition of the above morphisms.

Given a map $\rho: \pi_{1}\rightarrow G$ to a finite group, and a conjugacy-stable subset $C\subset G$, we seek to understand the proportion of points $v\in U(\mathbb{F}_{q^{w}})$ such that $\rho(F_{q^{w}, v})$ lies in $C$.

\begin{theorem}[Katz]
\label{katz}
Assume there is a commutative diagram

\[
\begin{tikzcd}
1 \arrow{r}  & \overline{\pi}_{1} \arrow {r} \arrow{d}[swap]{\overline{\rho}} & \pi_{1} \arrow{r} \arrow{d}{\rho} & \hat{\mathbb{Z}} \arrow{r} \arrow{d}{1\mapsto -\gamma} & 1 \\
1 \arrow{r} & \overline{G} \arrow{r} & G \arrow{r}{\mu} & \Gamma \arrow{r} & 1
\end{tikzcd}
\]
where $G$ is a finite group, $\Gamma$ is abelian, $\overline{\rho}$ is surjective and tamely ramified. Let $C\subset G$ be stable under conjugation by elements of $G$. Then
\begin{equation}
\label{maineq}
\left\vert \frac{\# \{v\in U(\mathbb{F}_{q^{w}}) \mid \rho(F_{q^{w}, v})\in C\}}{\#U(\mathbb{F}_{q^{w}})}-\frac{\#(C\cap G^{\gamma^{w}})}{\# \overline{G}}  \right\vert \,\leq\, \vert \chi(U)\vert \frac{\# G \sqrt{q^{w}}}{\#U(\mathbb{F}_{q^{w}})}\,,
\end{equation}
where $G^{\gamma^{w}}=\mu^{-1}(\gamma^{w})$ and $\chi(U)=\sum_{i=0}^{1}(-1)^{i}\dim \mathrm{H}^{i}(U, \mathbb{Q}_{\ell})$ is the $\ell$-adic Euler-Poincar\'e characteristic of $U$.
\end{theorem}
\begin{proof}
    See \cite[Theorem 4.1]{chav}.
\end{proof}
With the above in mind, we can now prove our effective gcd theorem. We recall our assumptions. Let $\mathcal{X}\subset \mathbb{P}^{N}$ be a smooth, projective geometrically irreducible variety of dimension $n$ and degree $D$, over a number field $K$. Let $\mathfrak{p}$ be a prime of good reduction, write $\mathbb{F}_{q}:=\mathcal{O}_{K}/\mathfrak{p}$ and denote the variety $X/\mathbb{F}_{q}$ upon reduction. Let $(X_{t})_{t\in\mathbb{P}^{1}}$ be a Lefschetz pencil of hyperplane sections on $X$. Denote by $Z\subset \mathbb{P}^{1}$ the finite set of nodal fibres and by $U=\mathbb{P}^{1}\setminus Z$, the subscheme parameterising the smooth fibres. 

\begin{reptheorem}{thm:main}
There exists a polynomial $\Phi(x)\in \mathbb{Z}[x]$ independent of $D$, such that for any extension $\mathbb{F}_{Q}/ \mathbb{F}_{q}$ with
    $$
    [\mathbb{F}_Q: \mathbb{F}_{q}] >\Phi(D),
    $$
    we have for any $u_{1}, u_{2}\in U(\mathbb{F}_{Q})$ chosen uniformly at random,
    $$
    P_{n-1}(X/ \mathbb{F}_{Q}, T)=\mathrm{gcd}\left(P_{n-1}(X_{u_{1}}/\mathbb{F}_{Q}, T),  P_{n-1}(X_{u_{2}}/\mathbb{F}_{Q}, T) \right);
    $$
    with probability $>2/3$.
\end{reptheorem}
\begin{proof}
Let $\ell$ be a large enough prime such that the groups $\mathrm{H}^{i}(X, \mathbb{Z}_{\ell})$ are all torsion-free. We can choose $\ell$ to be $\Omega(D^{2^{4N^{2}}})$ by the proof of Corollary~\ref{cor:ell}. Consider now the locally constant sheaf $R^{1}\pi_{\star}\mathbb{Z}_{\ell}\vert_{U}$ on $U$. It has as subsheaf $\mathcal{E}^{\mathbb{Z}_{\ell}}$ the sheaf of vanishing cycles. Write $\mathcal{E}^{\ell}=\mathcal{E}^{\mathbb{Z}_{\ell}}\otimes \mathbb{F}_{\ell}$ for the locally constant sheaf of mod -- $\ell$ vanishing cycles. Let $\rho_{\ell}: \pi_{1}(U_{0}, u)\rightarrow \mathrm{GL}(s, \mathbb{F}_{\ell})$ be the associated representation, and denote by $\overline{\rho}_{\ell}:=\rho_{\ell}\vert \pi_{1}(U, u)$ the restriction to the geometric fundamental group. We begin by assuming that the sheaf $\mathcal{E}^{\mathbb{Z}_{\ell}}$ has big monodromy. Indeed by the results of \cite[4.4]{Weilii}, we know that the monodromy is either big or finite, with the latter only happening in the orthogonal case.\\ 

We begin with the case of symplectic monodromy, i.e., $n$ is even, and by Theorem~\ref{thm:bigm}, the image of $\overline{\rho}_{\ell}$ is $\mathrm{Sp}(s, \mathbb{F}_{\ell})$.  We seek to apply Theorem~\ref{katz} to this setup with $\overline{G}=\mathrm{Sp}(s, \mathbb{F}_{\ell})$. Let $\mathbb{F}_{Q}/\mathbb{F}_{q}$ be an extension where $Q:= q^{w}$ and choose $u_{1}\in U(\mathbb{F}_{Q})$ randomly. We estimate the number of $v\in U(\mathbb{F}_{Q})$ such that $P(\mathcal{E}_{v}/\mathbb{F}_{Q},T)$ is coprime to $f(T):=P(\mathcal{E}_{u_{1}}/\mathbb{F}_{Q},T)$. Write $\overline{f}(T):=f(T) \ \mathrm{mod} \ \ell$.
    
     Denote by $C\subset \mathrm{GSp}(2r, \mathbb{F}_{\ell})$ the subset of matrices with characteristic polynomial not coprime to $\overline{f}(T)$. It is stable under conjugation by elements from $\mathrm{GSp}(2r, \mathbb{F}_{\ell})$. Applying Theorem~\ref{katz} to $C$, we get
$$
 \frac{\#\{v\in U(\mathbb{F}_{Q}) \ \vert \ \rho_{\ell}(F_{Q, v})\in C \} }{\#U(\mathbb{F}_{Q})} \;\leq\; \frac{\#(C\cap \mathrm{GSp}(2r, \mathbb{F}_{\ell})^{\gamma^{w}})}{\#\mathrm{Sp}(2r, \mathbb{F}_{\ell})}+ \vert \chi(U)\vert \frac{ \#\mathrm{GSp}(2r, \mathbb{F}_{\ell})\sqrt{q^{w}}}{\#U(\mathbb{F}_{Q})} \;.
$$
By Lemma~\ref{symplectic} (since $\ell> 119r^2$), the first summand on the RHS is $\leq 1/4$. From the calculation\footnote{see \cite[\href{https://stacks.math.columbia.edu/tag/03RR}{Tag 03RR}]{stacks-project}} of the \'{e}tale cohomology of $U$ (the projective line with $\#Z$ punctures), we deduce that $\vert \chi(U)\vert\leq\#Z\leq D^{N+1}$. Further, we see that $s$, which is the dimension of the space of vanishing cycles, is bounded above by the sum of the Betti numbers of the hyperplane section of $X$, which by Theorem~\ref{thm:milnor}, is at most $ND(2D - 1)^{2N+1}$. Therefore, for $q^{w}>2D^{N+1}$, we have
$$
\vert \chi(U)\vert \frac{ \#\mathrm{GSp}(s, \mathbb{F}_{\ell})\sqrt{q^{w}}}{\#U(\mathbb{F}_{Q})} \;\leq\; D^{N+1}\ell^{2s^{2}+s+1}\frac{\sqrt{q^{w}}}{q^{w}-D^{N+1}} \;\leq\; D^{N+1} D^{2^{4N^{2}}\cdot 4 N^{2}D^{2}(2D)^{6N}}\frac{\sqrt{q^{w}}}{q^{w}/2} \;.
$$
In particular, if $$Q=q^{w} > \Omega\left(D^{2^{8N^{2}}\cdot N^{2} \cdot D^{4N}}\right),$$ we have
$$
\frac{\#\{v\in U(\mathbb{F}_{Q}) \ \vert \ \rho_{\ell}(F_{Q, v})\not\in C \} }{\#U(\mathbb{F}_{Q})} \;>\; 2/3 \,,
$$
which completes the proof for the symplectic case. \\ \\
Now, we deal with the big orthogonal case, i.e., $n$ is odd and the image of $\overline{\rho}_{\ell}$ is one of the subgroups $\overline{\mathrm{G}}$ of $\mathrm{O}(s, \mathbb{F}_{\ell})$ of index at most two in Theorem~\ref{thm:bigm}. \footnote{We may assume the orthogonal monodromy is big by the remark after Theorem~\ref{thm:bigm}.} Denote by $\mathrm{G}$ its extension by an appropriate subgroup of $\mathbb{F}_{\ell}^{*}$ via (\ref{exac}). Let $C'\subset \mathrm{GO}(V, \mathbb{F}_{\ell})$ be the subset of matrices with characteristic polynomial  having distinct roots. Then, applying Theorem~\ref{katz}, we see
 $$\frac{\#\{v\in U(\mathbb{F}_{Q}) \ \vert \ \rho_{\ell}(F_{Q, v})\in C' \} }{\#U(\mathbb{F}_{Q})}\geq \frac{\#(C'\cap \mathrm{G}^{\gamma^{w}})}{\#\overline{\mathrm{G}}} - \vert \chi(U)\vert \frac{ \#\mathrm{G}\sqrt{q^{w}}}{\#U(\mathbb{F}_{Q})}.$$
 By Lemma~\ref{orthlem}, the first term of the RHS can be maximised with growing $\ell$, and the error term is minimised similar to the symplectic case. Now, for another trial $v'\in U(\mathbb{F}_{Q})$ chosen uniformly at random, we maximise the probability of the associated characteristic polynomial being coprime to that of the earlier trial via a similar estimate using Proposition~\ref{orthprob}. 
\end{proof}

\section*{Acknowledgements}
We thank Saugata Basu for conversations regarding the bound in Corollary~\ref{cor:simple torsion bound}. We thank Alan Lauder and George Walker for making  the thesis \cite{walker} available. We thank Nitin Saxena, T.N. Venkataramana and Arvind Nair for discussions. We are grateful to the organisers of the `Mordell Conjecture 100 years later' conference at MIT and the Simons Foundation for travel support. Parts of this work were conceived during the conference. H.J.K. is supported by the AMS-Simons Travel Grant.  M.V. is supported by a C3iHub research fellowship.
\bibliographystyle{alpha}
\bibliography{gcd.bib}

\end{document}